\documentclass[11pt]{amsproc}
\usepackage{amscd}

\theoremstyle{plain}

\newtheorem{corollary}{Corollary}

\newtheorem{definition}{Definition}
\newtheorem{example}{Example}

\newtheorem{lemma}{Lemma}

\newtheorem{remark}{Remark}

\newtheorem{theorem}{Theorem}
\numberwithin{equation}{section}

\input{tcilatex}

\begin{document}
\title[On the Morse index ...]{On the Morse index of harmonic maps and
minimal immersions}
\author{Mohammed Benalili}
\address{University AbouBakr Belkaid\\
Faculty of Sciences \\
Dept. of Mathematics\\
B.P. 119\\
Tlemcen Algeria}
\email{m\_benalili@mail.univ-tlemcen.dz}
\author{Hafida Benallal}
\email{hafedabenallal@yahoo.fr}

\begin{abstract}
In this paper we are concerned with harmonic maps and minimal immersions
defined on compact Riemannian manifolds and with values in homogenous
strongly harmonic manifolds. We show some results on the Morse index by
varying these maps along suitable conformal vector fields. We obtain also
that they are global maxima on some subspaces of the eigenspaces
corresponding to the nonvanishing eigenvalues of the Laplacian operator on
the target manifolds.
\end{abstract}

\maketitle

\footnote{%
AMSC 2000: 58E20, 53C42.
}

\section{Introduction}

Let $(M$ $^{m},g)\;$and $(N^{n},h)$ be Riemannian manifolds of dimension $m$
and $n$, respectively. If we use local coordinates, the metric tensors of $%
M^{m}$ and $N^{n}$ will be written as

\begin{equation*}
(g_{\alpha \beta })_{\alpha ,\beta =1,...m}\text{,}
\end{equation*}%
and

\begin{equation*}
(h_{ij})_{i,j=1,...,n}\text{.}
\end{equation*}%
The inverse metric tensor is 
\begin{equation*}
(g^{\alpha \beta })_{\alpha \beta =1,...m,}=(g_{\alpha \beta
})^{-1}\;_{\alpha ,\beta =1,...m,}
\end{equation*}%
and 
\begin{equation*}
\left\vert g\right\vert =\det (g_{\alpha \beta }).
\end{equation*}

If \ $f:M^{m}\rightarrow N^{n}$ is a map of class $C^{1}$, its energy
density is given by

\begin{equation}
e(f)(x)=\frac{1}{2}g^{\alpha \beta }(x)h_{ij}(f(x))\frac{\partial f_{i}(x)}{%
\partial x_{\alpha }}\frac{\partial f_{j}(x)}{\partial x_{\beta }}  \label{1}
\end{equation}%
in local coordinates $(x_{1},..,x_{m})$ on $M^{m}$ and $(y_{1},...,y_{n})$
on $N^{n}$ where the Einstein convention summation is used$.$ Or as an
intrinsic quantity of Riemannian geometry of $M^{m}$ and $N^{n}$

\begin{equation*}
e(f)=\frac{1}{2}\left\langle df,df\right\rangle _{T^{\ast }M^{m}\otimes
f^{-1}TN^{n}}
\end{equation*}%
where $\left\langle .,.\right\rangle _{f^{-1}TN^{n}}$ is the pullback by $f$
of the metric tensor of $N^{n}$.

Then the energy of $f$ \ is simply

\begin{equation}
E(f)=\int_{M^{m}}e(f)dv_{g}  \label{2}
\end{equation}%
with $dv_{g}=\sqrt{\left\vert g\right\vert }dx_{1}...dx_{m}$ \ the volume
element of $M^{m}$ \ in local coordinates.

Let $w$\ be a vector field along $f$ \ that means that $w$ \ is a section of 
$f^{-1}TN^{n}$ the pullback of the tangent space $TN^{n}$ by the map $f.\; $%
In local coordinates

\begin{equation*}
w=w^{i}(x)\frac{\partial }{\partial y_{i}}
\end{equation*}%
$w$ induces a variation of $f$ \ given by

\begin{equation*}
f_{t}(x)=\exp (tw)of(x).
\end{equation*}%
Put 
\begin{equation*}
\phi _{t}^{w}=\exp (tw)
\end{equation*}%
and 
\begin{equation}
\psi =\phi _{t_{o}}^{w}of\text{.}  \label{2'}
\end{equation}%
We get the first variation formula for the energy functional 
\begin{equation}
\frac{d}{dt}E(f_{t})\mid _{t=t_{0}}=-\int_{M^{m}}\left\langle \text{trace}%
\nabla d\psi ,wo\psi \right\rangle _{f^{-1}TN^{n}}dv_{g}.  \label{3}
\end{equation}%
$\tau _{g}(\psi )=$trace$\nabla d\psi $ is the so called tension field of $%
\psi $, where $\nabla $ denotes the covariant derivative on the manifold $%
N^{n}$.

\begin{definition}
The map $f$ is called harmonic if and only if the tensor field $\tau (f)=0.$
i.e. $f$ is a critical point of the energy functional $E$.
\end{definition}

\begin{definition}
The volume of an immersion $f,$ from a Riemannian manifold $(M^{m},g)$ into $%
(N^{n},h)$, denoted by $V(f)$, is defined as the Riemannian volume of $M^{m}$
endowed with the Riemannian metric $f^{\ast }h$, the pull back of the metric 
$h$.
\end{definition}

\begin{definition}
The map $f$ \ will be said minimal immersion if it satisfies 
\begin{equation*}
\frac{d}{dt}V(f_{t})\mid _{t=0}=0
\end{equation*}%
for any vector field $w$ on $N^{n}$ along $\;f.$
\end{definition}

Many results on the Morse index of harmonic maps and minimal immersions from
a compact Riemannian manifolds into the Euclidean sphere have been obtained
by varying the energy functional along conformal vector fields on the
Euclidean sphere. For a survey on the theory of harmonic maps and minimal
immersions, we refer the reader to ( \cite{1}, \cite{4}, \cite{9}).

In (\cite{5}), A. El Soufi shows\textit{\ that for every harmonic map }$f$%
\textit{\ from an }$m$\textit{- dimensional\ compact Riemannian manifold (}$%
M,g$\textit{) into the Euclidean sphere }$S^{n}$\textit{\ which enjoys one
of the following properties:}

\textit{(i) The stress-energy tensor }$S_{g}^{o}(f)$\textit{\ is positive
everywhere and positive definite at least at one point in }$M$

\textit{(ii) }$S_{g}^{o}(f)$\textit{\ is positive everywhere on }$M$\textit{%
, }$f$\textit{\ is an immersion and }$f(M)$\textit{\ is not a totally
geodesic sphere of }$S^{n}$\textit{.}

\textit{Then the Morse index of }$f$\textit{, }$\ \ \ Ind_{E}(f$\textit{\ }$%
)\geq n+1$.

Also he proves

\textit{Let }$f$\textit{\ be a minimal immersion from a compact }$m-$\textit{%
dimensional manifold }$M^{m}$\textit{\ into the sphere }$S^{n}\;(n\geq 3).$%
\textit{\ Then two possibilities hold:}

\textit{(i) }$Ind_{V}(f)=n-m$\textit{\ , hence }$\phi (M^{m})$\textit{\ is a
totally geodesic sphere of }$S^{n}$

\textit{(ii) }$Ind_{V}\left( f\right) \geq n+1.$\textit{\ Where }$%
Ind_{V}\left( f\right) $\textit{\ stands for the Morse index of }$f$.

In this paper we extend some results on the index of harmonic maps and
minimal immersions obtained in (\cite{5}) to non necessary spherical cases.
Mainly, if $L$ denotes an appropriate $n+1$- dimensional subspace of the
eigenspace $V_{\lambda }$ corresponding to the nonvanishing eigenvalue $%
\lambda $ of the Laplacian operator on the target manifold $N^{n}$ and $%
L^{\perp }$ is the normal component of $L$, we show that the Morse index, $%
Ind_{E}(f)$, of a harmonic map $f$ defined on a compact Riemannian manifold $%
M^{m}$ and with values in homogenous strongly harmonic manifolds $N^{n}$
fulfills $Ind_{E}(f)\geq \dim L$ provided that the sectional curvature $K$
of the target manifold $N^{n}$ satisfies $K\geq \kappa >0$, where $\kappa $
is a constant, the stress-energy tensor $S_{g}^{o}(f)$ is positive definite
and $\lambda \leq \frac{n^{2}}{2}\kappa $. We also obtain that if $%
f:M^{m}\rightarrow N^{n}$ is a minimal isometric immersion not totally
geodesic from a compact Riemannian manifold into a homogenous strongly
harmonic Riemannian manifold of dimension $n\geq 3$\ with sectional
curvature $K$ satisfying $K\geq \kappa >0$, where $\kappa $ is a constant,
then the Morse index, $Ind_{V}(f)\geq \dim (L{}^{\bot })$ provided that $%
\lambda $ satisfies $\lambda \leq \frac{n^{2}}{2}\kappa $.\ Finally we prove
that harmonic maps whose source manifolds are compact and the target ones
are homogenous strongly harmonic are global maximum on the appropriate $%
(n+1) $-dimensional subspace $L$ of the eigenspace $V_{\lambda }$ provided
that the stress-energy tensor $S_{g}^{o}(f)$ is positive and isometric
minimal immersions from compact manifolds into homogenous strongly harmonic
ones are global maximum on the $L^{\perp }$.

\section{Conformal vector fields on strongly harmonic manifolds}

In this section we will construct conformal vector fields on strong harmonic
manifolds along which we will vary the energy and the volume functionals.

\subsection{\protect\bigskip Strongly harmonic manifolds}

In this section, we will construct conformal vector fields by mean of
gradients of the eigenfunctions to Laplacian operator on strongly harmonic
manifolds.

\begin{definition}
A compact Riemannian manifold ($N^{n},h$) is said to be strongly harmonic(
we shall say $SH$- manifold) if there exists a map $\Xi :$ $R_{+}\times
R_{+}^{\ast }\rightarrow R$ with the property that the fundamental solution
of the heat equation $K$ on $(N^{n},h)$ can be written as $K(x,y,t)=\Xi
(\rho (x,y),t)$ for every $x$ and $y$ in $N^{n}$, $t$ in $R_{+}^{\ast }$ and 
$\rho $ is the distance function on $N^{n}$.
\end{definition}

For an eigenvalue $\lambda _{\alpha }$ of the Laplacian $\Delta =-div(\nabla
_{h})$ on $(N^{n},h)$, set $V_{\alpha }=\left\{ f:\Delta f=\lambda _{\alpha
}f\right\} $ the eigenspace corresponding to $\lambda _{\alpha }$ and $%
N_{\alpha }=\dim V_{\alpha }$.

Let $\left\{ \varphi _{i}^{\alpha }\right\} _{i}$ be an orthonormal basis of 
$V_{\alpha }$ with respect the global scalar product $\left\langle \varphi
,\psi \right\rangle $ $=\int_{N^{n}}\varphi \psi dv_{h}$ where $dv_{h}$
denotes the Riemannian measure corresponding to the metric $h$.

We know from (\cite{2}) that the fundamental solution of the heat equation $%
K $, on the manifold $(N^{n},h)$ writes as

\begin{equation}
K(x,y,t)=\sum_{\alpha }e^{-\lambda _{\alpha }t}\sum_{i=1}^{N_{\alpha
}}\varphi _{i}^{\alpha }(x)\varphi _{i}^{\alpha }(y)\text{ for every }x,y%
\text{ in }N^{n}\text{ and }t\text{ in }R_{+}^{\ast }\text{.}  \label{8}
\end{equation}%
From \ (\ref{8}), we deduce that the compact manifold $(N^{n},h)$ is $SH$ if
for every eigenvalue $\lambda _{\alpha }$ there exists a map $\Xi _{\alpha }$%
: $R_{+}\rightarrow R$ with \ 
\begin{equation}
\sum_{i=1}^{N_{\alpha }}\varphi _{i}^{\alpha }(x)\varphi _{i}^{\alpha
}(y)=\Xi _{\alpha }(\rho (x,y))\text{ for every }x,y\text{ in }N^{n}\text{.}
\label{8'}
\end{equation}

\begin{example}
An important class of homogenous $SH$- manifold is the compact symmetric
spaces of rank one which we denote by CROSS. Among the CROSS spaces we quote
the real projective spaces $RP^{n}$, and the Euclidean spheres $S^{n}$ ( see 
\cite{3} ).
\end{example}

Let $(N^{n},h)$ be an homogeneous $SH$- manifold and put for any $y\in N^{n}$%
, $\Lambda (y)=\left( \varphi _{1}^{\alpha }(y),...,\varphi _{N_{\alpha
}}^{\alpha }(y)\right) $. By the relation (\ref{8'}) the values of $\Lambda $
are in the Euclidean sphere $S^{N_{\alpha }-1}$ centred at the origin of $%
R^{N_{\alpha }}$ and of radius $\left( \Xi _{\alpha }(0)Vol(N,g)\right) ^{%
\frac{1}{2}}$.We quote the following lemma

\begin{lemma}
\label{lem1}(\cite{3}) $\Lambda $ is an immersion and its image $\Lambda
(N^{n})$ is a $n$- dimensional submanifold of the Euclidean sphere $%
S^{N_{\alpha }-1}(0,R)$ in $R^{N_{\alpha }}$ of radius $R=\left( \Xi
_{\alpha }(0)Vol(N,g)\right) ^{\frac{1}{2}}$.
\end{lemma}

As a corollary of Lemma\ref{lem1}, we have

\begin{corollary}
For an $n$- dimensional SH- manifold every eigenvalue of the Laplacian has
multiplicity greater or equal to $n+1$.
\end{corollary}

\ Since $\Lambda $ is an immersion it is locally an embedding so we can
choose $n+1$ points $m_{1},...,m_{n+1}$ in $N^{n}$ such that their images $%
\Lambda (m_{1}),...,\Lambda (m_{n+1})$ are linear independent positions
vectors in $R^{N_{\alpha }}$. Consider the functions 
\begin{equation}
u_{j}=\sum_{i=1}^{N_{\alpha }}\varphi _{i}^{\alpha }(m_{j})\varphi
_{i}^{\alpha }\text{, }1\leq j\leq n+1\text{,}  \label{8''}
\end{equation}%
where $\left\{ \varphi _{i}^{\alpha }\right\} _{i}$ is an orthonormal basis
of the eigenspace $V_{\alpha }$, clearly $u_{j}$ are also eigenfunctions.
Put 
\begin{equation*}
\Xi _{\alpha }(\rho (m_{j},m_{k}))=\sum_{i=1}^{N_{\alpha }}\varphi
_{i}^{\alpha }(m_{j})\varphi _{i}^{\alpha }(m_{k})\text{.}
\end{equation*}%
Now, we will state the following lemma which will be crucial in the sequel
of this paper.

\begin{lemma}
The matrix $\left( \Xi _{\alpha }(\rho (m_{j},m_{k})\right) _{1\leq j,k\leq
n+1}$ is invertible.
\end{lemma}

\begin{proof}
Let us first show that the determinant of the matrix $\left( \Xi _{\alpha
}(\rho (m_{j},m_{k})\right) _{1\leq j,k\leq n+1}$ enjoys 
\begin{equation}
\det \left( \Xi _{\alpha }(\rho (m_{j},m_{k}))\right) =\frac{1}{\left(
n+1\right) !}\sum_{i_{1},i_{2},...,i_{n+1}=1}^{N_{\alpha }}\det \left( 
\begin{array}{ccc}
\varphi _{i_{1}}^{\alpha }(m_{1}) & ... & \varphi _{i_{1}}^{\alpha }(m_{n+1})
\\ 
\vdots &  & \vdots \\ 
\varphi _{_{i_{n+1}}}^{\alpha }(m_{1}) & ... & \varphi _{_{i_{n+1}}}^{\alpha
}(m_{n+1})%
\end{array}%
\right) ^{2}\text{.}  \label{8"}
\end{equation}%
We proceed by recurrence. In the case $n+1=2$, the determinant of the
matrix\ $\left( \Xi _{\alpha }(\rho (m_{j},m_{k})\right) _{1\leq j,k\leq 2}$
is given by%
\begin{equation*}
\det \left( \Xi _{\alpha }(\rho (m_{j},m_{k}))\right)
=\sum_{i_{1}=1}^{N_{\alpha }}\varphi _{i_{1}}^{\alpha
}(m_{1})^{2}\sum_{i_{2}=1}^{N_{\alpha }}\varphi _{i_{2}}^{\alpha
}(m_{2})^{2}-\left( \sum_{i_{1},i_{2}=1}^{N_{\alpha }}\varphi
_{i_{1}}^{\alpha }(m_{1})\varphi _{i_{2}}^{\alpha }(m_{2})\right) ^{2}
\end{equation*}%
\begin{equation*}
=\sum_{i_{1}\neq i_{2}}\varphi _{i_{1}}^{\alpha }(m_{1})^{2}\varphi
_{_{i_{2}}}^{\alpha }(m_{2})^{2}-2\sum_{i_{1}<i_{2}}\varphi _{i_{1}}^{\alpha
}(m_{1})\varphi _{i_{1}}^{\alpha }(m_{2})\varphi _{i_{2}}^{\alpha
}(m_{1})\varphi _{i_{2}}^{\alpha }(m_{2})
\end{equation*}%
\begin{equation*}
=\sum_{1\leq i_{1}<i_{2}\leq N_{\alpha }}\left( \varphi _{i_{1}}^{\alpha
}(m_{1})\varphi _{_{i_{2}}}^{\alpha }(m_{2})-\varphi _{i_{1}}^{\alpha
}(m_{2})\varphi _{_{i_{2}}}^{\alpha }(m_{1})\right) ^{2}
\end{equation*}%
\begin{equation}
=\frac{1}{2}\sum_{i_{1},i_{2}=1}^{N_{\alpha }}\det \left( 
\begin{array}{cc}
\varphi _{i_{1}}^{\alpha }(m_{1}) & \varphi _{i_{1}}^{\alpha }(m_{2}) \\ 
\varphi _{_{i_{2}}}^{\alpha }(m_{1}) & \varphi _{_{i_{2}}}^{\alpha }(m_{2})%
\end{array}%
\right) ^{2}\text{.}  \label{8'''}
\end{equation}%
For any integer $2\leq p<n+1$, suppose that%
\begin{equation*}
\det \left( \Xi _{\alpha }(\rho (m_{j},m_{k}))\right) =\frac{1}{p!}%
\sum_{i_{1},...,i_{p=1}}^{N_{\alpha }}\det \left( 
\begin{array}{ccc}
\varphi _{i_{1}}^{\alpha }(m_{1}) & \cdots & \varphi _{i_{1}}^{\alpha
}(m_{p}) \\ 
\vdots &  & \vdots \\ 
\varphi _{i_{p}}^{\alpha }(m_{2}) & \cdots & \varphi _{i_{p}}^{\alpha
}(m_{p})%
\end{array}%
\right) ^{2}
\end{equation*}%
We have 
\begin{equation*}
\det \left( \Xi _{\alpha }(\rho (m_{j},m_{k}))\right) =\sum_{\sigma \in
S_{p+1}}sgn(\sigma )\Xi _{\alpha }(\rho (m_{1},m_{\sigma (1)})...\Xi
_{\alpha }(\rho (m_{p+1},m_{\sigma (p+1)})
\end{equation*}%
\begin{equation*}
=\sum_{i_{1},...,i_{p+1}=1}^{N_{\alpha }}\sum_{\sigma \in Sp+1}sgn(\sigma
)\varphi _{i_{1}}^{\alpha }(m_{1})\varphi _{i_{1}}^{\alpha }(m_{\sigma
(1)})...\varphi _{i_{p+1}}^{\alpha }(m_{p+1})\varphi _{i_{p+1}}^{\alpha
}(m_{\sigma (p+1)})
\end{equation*}%
\begin{equation*}
=\sum_{i_{1},...,i_{p+1}=1}^{N_{\alpha }}\varphi _{i_{1}}^{\alpha
}(m_{1})^{2}\sum_{\sigma \in S_{p+1}}sgn(\sigma )\varphi _{i_{2}}^{\alpha
}(m_{2})\varphi _{i_{2}}^{\alpha }(m_{\sigma (2)})...\varphi
_{i_{p+1}}^{\alpha }(m_{p+1})\varphi _{i_{p+1}}^{\alpha }(m_{\sigma (p+1)})
\end{equation*}%
\begin{equation}
+\sum_{i_{1},...,i_{p+1}=1}^{N_{\alpha }}\sum_{\sigma \in Sp+1,\sigma
(1)\neq 1}sgn(\sigma )\varphi _{i_{1}}^{\alpha }(m_{1})\varphi
_{i_{1}}^{\alpha }(m_{\sigma (1)})...\varphi _{i_{p+1}}^{\alpha
}(m_{p+1})\varphi _{i_{p+1}}^{\alpha }(m_{\sigma (p+1)}  \label{9'}
\end{equation}%
where $S_{p+1}$ denotes the finite cyclic group of cardinal $p+1$ and $%
sgn(\sigma )$ stands for the sign of the permutation $\sigma $.

By the recurrent hypothesis the first term of the right hand side of the
equality (\ref{9'}) becomes%
\begin{equation*}
\sum_{i_{1},...,i_{p+1}=1}^{N_{\alpha }}\varphi _{i_{1}}^{\alpha
}(m_{1})^{2}\sum_{\sigma \in S_{p+1}}sgn(\sigma )\varphi _{i_{2}}^{\alpha
}(m_{2})\varphi _{i_{2}}^{\alpha }(m_{\sigma (2)})...\varphi
_{i_{p+1}}^{\alpha }(m_{p+1})\varphi _{i_{p+1}}^{\alpha }(m_{\sigma (p+1)})=
\end{equation*}%
\begin{equation*}
=\frac{1}{p!}\sum_{i_{1},...,i_{p+1}=1}^{N_{\alpha }}\varphi
_{i_{1}}^{\alpha }(m_{1})^{2}\det \left( 
\begin{array}{ccc}
\varphi _{i_{2}}^{\alpha }(m_{2}) & \cdots & \varphi _{i_{2}}^{\alpha
}(m_{p+1}) \\ 
\vdots &  & \vdots \\ 
\varphi _{i_{p+1}}^{\alpha }(m_{2}) & \cdots & \varphi _{i_{p+1}}^{\alpha
}(m_{p+1})%
\end{array}%
\right) ^{2}
\end{equation*}%
\begin{equation*}
=\frac{1}{\left( p+1\right) !}\sum_{i_{1},...,i_{p+1}=1}^{N_{\alpha
}}\varphi _{i_{1}}^{\alpha }(m_{1})^{2}\left( \det \left( 
\begin{array}{ccc}
\widehat{\varphi _{i_{1}}^{\alpha }(m_{2})} & \cdots & \widehat{\varphi
_{i_{1}}^{\alpha }(m_{p+1})} \\ 
\varphi _{i_{2}}^{\alpha }(m_{2}) & \cdots & \varphi _{i_{2}}^{\alpha
}(m_{p+1}) \\ 
\vdots &  & \vdots \\ 
\varphi _{i_{p+1}}^{\alpha }(m_{2}) & \cdots & \varphi _{i_{p+1}}^{\alpha
}(m_{p+1})%
\end{array}%
\right) ^{2}+\cdots \right.
\end{equation*}%
\begin{equation}
\left. +\varphi _{i_{p+1}}^{\alpha }(m_{1})^{2}\det \left( 
\begin{array}{ccc}
\varphi _{i_{1}}^{\alpha }(m_{2}) & \cdots & \varphi _{i_{1}}^{\alpha
}(m_{p+1}) \\ 
\vdots &  & \vdots \\ 
\varphi _{i_{p}}^{\alpha }(m_{2}) & \cdots & \varphi _{i_{p}}^{\alpha
}(m_{p+1}) \\ 
\widehat{\varphi _{i_{p+1}}^{\alpha }(m_{2})} & \cdots & \widehat{\varphi
_{i_{p+1}}^{\alpha }(m_{p+1})}%
\end{array}%
\right) ^{2}\right)  \label{9''}
\end{equation}%
where the writing $\widehat{x}$ means that $x$ does not appear.

The last term of the right hand side of (\ref{9'}) writes%
\begin{equation*}
\sum_{i_{1},...,i_{p+1}=1}^{N_{\alpha }}\sum_{\sigma \in Sp+1,\sigma (1)\neq
1}sgn(\sigma )\varphi _{i_{1}}^{\alpha }(m_{1})\varphi _{i_{1}}^{\alpha
}(m_{\sigma (1)})...\varphi _{i_{p+1}}^{\alpha }(m_{p+1})\varphi
_{i_{p+1}}^{\alpha }(m_{\sigma (p+1)}=
\end{equation*}%
\begin{equation*}
\sum_{i_{1},...,i_{p+1}=1}^{N_{\alpha }}\varphi _{i_{1}}^{\alpha
}(m_{1})...\varphi _{i_{p+1}}^{\alpha }(m_{p+1})\left\{ -\varphi
_{i_{2}(m_{1})}\sum_{\sigma \in Sp+1,\sigma (1)\neq 1}sgn(\sigma )\varphi
_{i_{1}}^{\alpha }(m_{\sigma (1)})...\varphi _{i_{p+1}}^{\alpha }(m_{\sigma
(p+1)}+\cdots \right.
\end{equation*}%
\begin{equation*}
\left. +(-1)^{p}\varphi _{i_{p+1}}^{\alpha }(m_{1})\sum_{\sigma \in
Sp+1,\sigma (1)\neq 1}sgn(\sigma )\varphi _{i_{1}}^{\alpha }(m_{\sigma
(1)})...\varphi _{i_{p+1}}^{\alpha }(m_{\sigma (p+1)}\right\} =
\end{equation*}%
\begin{equation*}
=\sum_{i_{1},...,i_{p+1}=1}^{N_{\alpha }}\varphi _{i_{1}}^{\alpha
}(m_{1})...\varphi _{i_{p+1}}^{\alpha }(m_{p+1})\left\{ -\varphi
_{i_{1}(m_{1})}^{\alpha }\det \left( 
\begin{array}{ccc}
\varphi _{i_{1}}^{\alpha }(m_{2}) & \cdots & \varphi _{i_{1}}^{\alpha
}(m_{p+1}) \\ 
\widehat{\varphi _{i_{2}}^{\alpha }(m_{2})} & \cdots & \widehat{\varphi
_{i_{2}}^{\alpha }(m_{p+1})} \\ 
\vdots &  & \vdots \\ 
\varphi _{i_{p+1}}^{\alpha }(m_{2}) & \cdots & \varphi _{i_{p+1}}^{\alpha
}(m_{p+1})%
\end{array}%
\right) +\cdots \right.
\end{equation*}%
\begin{equation*}
\left. +(-1)^{p}\varphi _{i_{p+1}}^{\alpha }(m_{1})\det \left( 
\begin{array}{ccc}
\varphi _{i_{1}}^{\alpha }(m_{2}) & \cdots & \varphi _{i_{1}}^{\alpha
}(m_{p+1}) \\ 
\vdots &  & \vdots \\ 
\varphi _{i_{p+1}}^{\alpha }(m_{2}) & \cdots & \varphi _{i_{p+1}}^{\alpha
}(m_{p+1}) \\ 
\widehat{\varphi _{i_{p+1}}^{\alpha }(m_{2})} & \cdots & \widehat{\varphi
_{i_{p+1}}^{\alpha }(m_{p+1})}%
\end{array}%
\right) \right\} =
\end{equation*}%
By the definition of the determinant, we get%
\begin{equation*}
=\frac{1}{p!}\sum_{i_{1},...,i_{p+1}=1}^{N_{\alpha }}-\varphi
_{i_{1}}^{\alpha }(m_{1})\varphi _{i_{2}}^{\alpha }(m_{1})\sum_{\sigma \in
S_{p},\sigma (2)\neq 1}\varphi _{i_{2}}^{\alpha }(m_{\sigma (2)})...\varphi
_{i_{p+1}}^{\alpha }(m_{\sigma (p+1)})
\end{equation*}%
\begin{equation*}
\times \det \left( 
\begin{array}{ccc}
\varphi _{i_{1}}^{\alpha }(m_{\sigma (2)}) & \cdots & \varphi
_{i_{1}}^{\alpha }\left( m_{\sigma (p+1)}\right) \\ 
\widehat{\varphi _{i_{2}}^{\alpha }(m_{\sigma (2)})} & \cdots & \widehat{%
\varphi _{i_{2}}^{\alpha }(m_{\sigma (p+1)})} \\ 
\vdots &  & \vdots \\ 
\varphi _{i_{p+1}}^{\alpha }(m_{\sigma (2)}) & \cdots & \varphi
_{i_{p+1}}^{\alpha }(m_{\sigma (p+1)})%
\end{array}%
\right) +\cdots
\end{equation*}%
\begin{equation*}
+\frac{1}{p!}\sum_{i_{1},...,i_{p+1}=1}^{N_{\alpha }}\left( -1\right)
^{p}\varphi _{i_{1}}^{\alpha }(m_{1})\varphi _{i_{p+1}}^{\alpha
}(m_{1})\sum_{\sigma \in S_{p},\sigma (p+1)\neq 1}\varphi _{i_{2}}^{\alpha
}(m_{\sigma (2)})...\varphi _{i_{p+1}}^{\alpha }(m_{\sigma (p+1)})
\end{equation*}%
\begin{equation*}
\times \det \left( 
\begin{array}{ccc}
\varphi _{i_{1}}^{\alpha }(m_{\sigma (2)}) & \cdots & \varphi
_{i_{1}}^{\alpha }\left( m_{\sigma (p+1)}\right) \\ 
&  &  \\ 
\varphi _{i_{p}}^{\alpha }(m_{\sigma (2)}) & \cdots & \varphi
_{i_{p}}^{\alpha }(m_{\sigma (p+1)}) \\ 
\widehat{\varphi _{i_{p+1}}(m_{\sigma (2)})} & \cdots & \widehat{\varphi
_{i_{p+1}}(m_{\sigma (p+1)})}%
\end{array}%
\right)
\end{equation*}%
Now, by permuting the indices $(i_{1},...,i_{p+1})$ and noting that it turns
out to permute $(m_{2},...,m_{p+1})$, we deduce%
\begin{equation*}
=\frac{1}{p!}\sum_{i_{1},...,i_{p+1}=1}^{N_{\alpha }}-\varphi
_{i_{1}}^{\alpha }(m_{1})\varphi _{i_{2}}^{\alpha }(m_{1})\sum_{\sigma \in
S_{p},\sigma (2)\neq 1}sgn(\sigma )\varphi _{i_{2}}^{\alpha }(m_{\sigma
(2)})...\varphi _{i_{p+1}}^{\alpha }(m_{\sigma (p+1)})
\end{equation*}%
\begin{equation*}
\times \det \left( 
\begin{array}{ccc}
\varphi _{i_{1}}^{\alpha }(m_{2}) & \cdots & \varphi _{i_{1}}^{\alpha
}\left( m_{p+1}\right) \\ 
\widehat{\varphi _{i_{2}}^{\alpha }(m_{2})} & \cdots & \widehat{\varphi
_{i_{2}}^{\alpha }(m_{p+1})} \\ 
\vdots &  & \vdots \\ 
\varphi _{i_{p+1}(m_{2})}^{\alpha } & \cdots & \varphi
_{i_{p+1}(m_{p+1})}^{\alpha }%
\end{array}%
\right) +\cdots
\end{equation*}%
\begin{equation*}
+\frac{1}{p!}\sum_{i_{1},...,i_{p+1}=1}^{N_{\alpha }}\left( -1\right)
^{p}\varphi _{i_{1}}^{\alpha }(m_{1})\varphi _{i_{p+1}}^{\alpha
}(m_{1})\sum_{\sigma \in S_{p},\sigma (p+1)\neq 1}sgn(\sigma )\varphi
_{i_{2}}^{\alpha }(m_{\sigma (2)})...\varphi _{i_{p+1}}^{\alpha }(m_{\sigma
(p+1)})
\end{equation*}%
\begin{equation*}
\times \det \left( 
\begin{array}{ccc}
\varphi _{i_{1}}^{\alpha }(m_{2}) & \cdots & \varphi _{i_{1}}^{\alpha
}\left( m_{p+1}\right) \\ 
\vdots &  & \vdots \\ 
\varphi _{i_{p}}^{\alpha }(m_{2}) & \cdots & \varphi _{i_{p}}^{\alpha
}(m_{p+1}) \\ 
\widehat{\varphi _{i_{p+1}}(m_{2})} & \cdots & \widehat{\varphi
_{i_{p+1}}(m_{p+1})}%
\end{array}%
\right)
\end{equation*}%
\begin{equation*}
=\frac{1}{p!}\sum_{i_{1},...,i_{p+1}=1}^{N_{\alpha }}\left( -1\right)
\varphi _{i_{1}}^{\alpha }(m_{1})\varphi _{i_{2}}^{\alpha }(m_{1})\det
\left( 
\begin{array}{ccc}
\varphi _{i_{1}}^{\alpha }(m_{2}) & \cdots & \varphi _{i_{1}}^{\alpha
}\left( m_{p+1}\right) \\ 
\widehat{\varphi _{i_{2}}^{\alpha }(m_{2})} & \cdots & \widehat{\varphi
_{i_{2}}^{\alpha }(m_{p+1})} \\ 
\vdots &  & \vdots \\ 
\varphi _{i_{p+1}}^{\alpha }(m_{2}) & \cdots & \varphi _{i_{p+1}}^{\alpha
}(m_{p+1})%
\end{array}%
\right)
\end{equation*}%
\begin{equation*}
\times \det \left( 
\begin{array}{ccc}
\widehat{\varphi _{i_{1}}^{\alpha }(m_{2})} & \cdots & \widehat{\varphi
_{i_{1}}^{\alpha }\left( m_{p+1}\right) } \\ 
\varphi _{i_{2}(m_{2})}^{\alpha } & \cdots & \varphi
_{i_{2}(m_{p+1})}^{\alpha } \\ 
\vdots &  & \vdots \\ 
\varphi _{i_{p+1}(m_{2})}^{\alpha } & \cdots & \varphi
_{i_{p+1}(m_{p+1})}^{\alpha }%
\end{array}%
\right) +\cdots
\end{equation*}%
\begin{equation*}
+\frac{1}{p!}\sum_{i_{1},...,i_{p+1}=1}^{N_{\alpha }}\left( -1\right)
^{p}\varphi _{i_{1}}^{\alpha }(m_{1})\varphi _{i_{p+1}}^{\alpha }(m_{1})\det
\left( 
\begin{array}{ccc}
\widehat{\varphi _{i_{1}}^{\alpha }(m_{2})} & \cdots & \widehat{\varphi
_{i_{1}}^{\alpha }\left( m_{p+1}\right) } \\ 
\varphi _{i_{2}}^{\alpha }(m_{2}) & \cdots & \varphi _{i_{2}}^{\alpha
}(m_{p+1}) \\ 
\vdots &  & \vdots \\ 
\varphi _{i_{p+1}}^{\alpha }(m_{2}) & \cdots & \varphi _{i_{p+1}}^{\alpha
}(m_{p+1})%
\end{array}%
\right)
\end{equation*}%
\begin{equation*}
\times \det \left( 
\begin{array}{ccc}
\varphi _{i_{1}}^{\alpha }(m_{2}) & \cdots & \varphi _{i_{1}}^{\alpha
}\left( m_{p+1}\right) \\ 
\varphi _{i_{2}}^{\alpha }(m_{2}) & \cdots & \varphi _{i_{2}}^{\alpha
}(m_{p+1}) \\ 
\vdots &  & \vdots \\ 
\widehat{\varphi _{i_{p+1}}^{\alpha }(m_{2})} & \cdots & \widehat{\varphi
_{i_{p+1}}^{\alpha }(m_{p+1})}%
\end{array}%
\right)
\end{equation*}%
\begin{equation*}
=\frac{1}{\left( p-1\right) !}\sum_{i_{1},...,i_{p+1}=1}^{N_{\alpha }}\left(
-1\right) \varphi _{i_{1}}^{\alpha }(m_{1})\varphi _{i_{2}}^{\alpha
}(m_{1})\det \left( 
\begin{array}{ccc}
\varphi _{i_{1}}^{\alpha }(m_{2}) & \cdots & \varphi _{i_{1}}^{\alpha
}\left( m_{p+1}\right) \\ 
\widehat{\varphi _{i_{2}}^{\alpha }(m_{2})} & \cdots & \widehat{\varphi
_{i_{2}}^{\alpha }(m_{p+1})} \\ 
\vdots &  & \vdots \\ 
\varphi _{i_{p+1}}^{\alpha }(m_{2}) & \cdots & \varphi _{i_{p+1}}^{\alpha
}(m_{p+1})%
\end{array}%
\right)
\end{equation*}%
\begin{equation*}
\times \det \left( 
\begin{array}{ccc}
\widehat{\varphi _{i_{1}}^{\alpha }(m_{2})} & \cdots & \widehat{\varphi
_{i_{1}}^{\alpha }\left( m_{p+1}\right) } \\ 
\varphi _{i_{2}}^{\alpha }(m_{2}) & \cdots & \varphi _{i_{2}}^{\alpha
}(m_{p+1}) \\ 
\vdots &  & \vdots \\ 
\varphi _{i_{p+1}}^{\alpha }(m_{2}) & \cdots & \varphi _{i_{p+1}}^{\alpha
}(m_{p+1})%
\end{array}%
\right)
\end{equation*}%
\begin{equation}
=\frac{2}{\left( p+1\right) !}\sum_{i_{1},...,i_{p+1}=1}^{N_{\alpha
}}\sum_{1\leq j<l\leq p+1}\left( -1\right) ^{j+l}\varphi _{i_{j}}^{\alpha
}(m_{1})\varphi _{i_{l}}^{\alpha }(m_{1})\det \left( 
\begin{array}{ccc}
\varphi _{i_{1}}^{\alpha }(m_{2}) & \cdots & \varphi _{i_{1}}^{\alpha
}\left( m_{p+1}\right) \\ 
\vdots &  & \vdots \\ 
\widehat{\varphi _{i_{l}}^{\alpha }(m_{2})} & \cdots & \widehat{\varphi
_{i_{l}}^{\alpha }(m_{p+1})} \\ 
\vdots &  & \vdots \\ 
\varphi _{i_{p+1}}^{\alpha }(m_{2}) & \cdots & \varphi _{i_{p+1}}^{\alpha
}(m_{p+1})%
\end{array}%
\right) .  \label{9"}
\end{equation}%
\begin{equation*}
\times \det \left( 
\begin{array}{ccc}
\varphi _{i_{1}}^{\alpha }(m_{2}) & \cdots & \varphi _{i_{1}}^{\alpha
}\left( m_{p+1}\right) \\ 
\vdots &  & \vdots \\ 
\widehat{\varphi _{i_{j}}^{\alpha }(m_{2})} & \cdots & \widehat{\varphi
_{i_{j}}^{\alpha }(m_{p+1})} \\ 
\vdots &  & \vdots \\ 
\varphi _{i_{p+1}}^{\alpha }(m_{2}) & \cdots & \varphi _{i_{p+1}}^{\alpha
}(m_{p+1})%
\end{array}%
\right) .
\end{equation*}%
Reporting (\ref{9''}) and (\ref{9"}) in (\ref{9'}), we obtain the equality(%
\ref{8"}).

Now, Since by construction the matrix $\left( \varphi _{j}^{\alpha
}(m_{k})\right) _{\substack{ 1\leq j\leq N_{\alpha }  \\ 1\leq k\leq n+1}}$
is of rank $n+1$, there is a cofactor $\left( \varphi _{j}^{\alpha
}(m_{k})\right) _{1\leq j,k\leq n+1}$ with determinant $\det \left( \varphi
_{j}^{\alpha }(m_{k})\right) \neq 0$, and from the relation (\ref{8"}), we
deduce that $\det \left( \Xi _{\alpha }(\rho (m_{j},m_{k}))\right) \neq 0$.
\end{proof}

\begin{lemma}
The gradient vector fields $\nabla u_{j}$, $j=1,...,n+1$ are linearly
independent.
\end{lemma}

\begin{proof}
Let $\xi _{j}\in R$ be real numbers such that $\sum_{j=1}^{n+1}\xi
_{j}\nabla u_{j}=0.$ So 
\begin{equation*}
0=\sum_{j=1}^{n+1}\xi _{j}\int_{N^{n}}\left\langle \nabla u_{j},\nabla
u_{k}\right\rangle dv_{h}=\sum_{j=1}^{n+1}\xi _{j}\int_{N^{n}}\Delta
u_{j}u_{k}dv_{h}
\end{equation*}%
\begin{equation*}
=\lambda _{\alpha }\sum_{j=1}^{n+1}\xi _{j}\int_{N^{n}}u_{j}u_{k}dv_{h}
\end{equation*}%
\begin{equation*}
=\lambda _{\alpha }\sum_{j=1}^{n+1}\xi _{j}\sum_{i=1}^{N_{\alpha
}}\sum_{l=1}^{N_{\alpha }}\int_{N^{n}}\varphi _{i}^{\alpha }(m_{j})\varphi
_{l}^{\alpha }(m_{k})dv_{h}
\end{equation*}%
\begin{equation*}
=\lambda _{\alpha }\sum_{j=1}^{n+1}\xi _{j}\sum_{i=1}^{N_{\alpha }}\varphi
_{i}^{\alpha }(m_{j})\varphi _{i}^{\alpha }(m_{k})
\end{equation*}%
and since%
\begin{equation*}
\sum_{i=1}^{N_{\alpha }}\varphi _{i}^{\alpha }(m_{j})\varphi _{i}^{\alpha
}(m_{k})=\Xi (\rho (m_{j},m_{k}))
\end{equation*}%
we get 
\begin{equation*}
\sum_{j=1}^{n+1}\xi _{j}.\Xi _{\alpha }(\rho (m_{j},m_{k}))=0
\end{equation*}%
so if the matrix $\Xi (\rho (m_{j},m_{k}))_{1\leq j,k\leq n+1}$ is
invertible, we get that $\xi _{j}=0$ for $j=1,...,n+1$.
\end{proof}

\begin{lemma}
The gradient vector fields $\nabla u_{j}$, $j=1,...,n+1$ are conformal
vector fields\ on $(N^{n},h).$
\end{lemma}

\begin{proof}
Let $X,$ $Y$ any orthogonal vector fields on $N^{n}$ i.e. $h(X,Y)=0$ and $%
Z=\nabla u_{j}$. We have to show that the Lie derivative $L_{Z}(h$ ) of the
tensor metric $h$ with respect to $Z$ satisfies 
\begin{equation*}
(L_{Z}h)(X,Y)=0\text{.}
\end{equation*}%
But 
\begin{equation*}
(L_{Z}h)(X,Y)=h(\nabla _{X}Z,Y)+h(\nabla _{Y}Z,X)
\end{equation*}

\begin{equation}
=2h(\nabla _{X}Z,Y)=2Hess(u_{j})(X,Y)  \label{9}
\end{equation}%
and since the manifold $(N^{n},h)$ is $SH$,\ there is a function $\Xi
_{\alpha }$: $R_{+}\rightarrow R$ with 
\begin{equation}
\sum_{i=1}^{N_{\alpha }}\varphi _{i}^{\alpha }(x)\varphi _{i}^{\alpha
}(y)=\Xi _{\alpha }(\rho (x,y))  \label{10}
\end{equation}%
for every $x,y$ in $N^{n}$.

Putting $x=y$ in (\ref{10}) and differentiating twice we get 
\begin{equation}
\sum_{i=1}^{N_{\alpha }}\varphi _{i}^{\alpha }(y)d\varphi _{i}^{\alpha }(y)=0%
\text{, }\sum_{i=1}^{N_{\alpha }}d\varphi _{i}^{\alpha }(y)\otimes d\varphi
_{i}^{\alpha }(y)+\sum_{i=1}^{N_{\alpha }}\varphi _{i}^{\alpha
}(y)Hess(\varphi _{i}^{\alpha })(y)=0\text{.}  \label{11}
\end{equation}%
Now differentiating (\ref{10}) twice with respect to $x$ with $y$ fixed we
obtain 
\begin{equation}
\sum_{i=1}^{N_{\alpha }}\varphi _{i}^{\alpha }(y)Hess(\varphi _{i}^{\alpha
})(x)=Hess(\Xi _{\alpha }(\rho (,y))(x)  \label{12}
\end{equation}%
and evaluating at $x=y$ \ we get 
\begin{equation*}
Hess(\Xi (\rho (.,y))(y)=\Xi _{\alpha }^{\prime \prime }(0)h.
\end{equation*}%
By taking account of (\ref{11}), we obtain 
\begin{equation*}
\sum_{i=1}^{N_{\alpha }}d\varphi _{i}^{\alpha }(x)\otimes d\varphi
_{i}^{\alpha }(x)=-\Xi _{\alpha }^{\prime \prime }(0)h.
\end{equation*}%
To compute $\Xi _{\alpha }^{\prime \prime }(0)$, we take the traces in (\ref%
{12}) and infer that 
\begin{equation*}
\Xi _{\alpha }^{\prime \prime }(0)(y)n=\sum_{i=1}^{N_{\alpha }}\varphi
_{i}^{\alpha }(y)\Delta \varphi _{i}^{\alpha }(y)=-\lambda _{\alpha
}\sum_{i=1}^{N_{\alpha }}(\varphi _{i}^{\alpha }(y))^{2}
\end{equation*}%
\begin{equation*}
=-\lambda _{\alpha }\Xi _{\alpha }(\rho (y,y))=-\lambda _{\alpha }\Xi
_{\alpha }(0)
\end{equation*}%
where $n$ is the dimension of the manifold $N^{n}$.

So 
\begin{equation*}
\Xi _{\alpha }^{\prime \prime }(0)(y)=-\frac{\lambda _{\alpha }}{n}\Xi
_{\alpha }(\rho (y,y))
\end{equation*}%
and 
\begin{equation*}
\sum_{i=1}^{N_{\alpha }}d\varphi _{i}^{\alpha }\otimes d\varphi _{i}^{\alpha
}=\frac{\lambda _{\alpha }}{n}\Xi _{\alpha }(0)h.
\end{equation*}%
By (\ref{11}), we get 
\begin{equation}
\sum_{i=1}^{N_{\alpha }}\varphi _{i}^{\alpha }(y)\left( Hess(\varphi
_{i}^{\alpha })(y)+\frac{\lambda _{\alpha }}{n}\varphi _{i}^{\alpha
}(y)h\right) =0\text{.}  \label{13}
\end{equation}

On the other hand, for the functions $u_{j}$ defined by (\ref{8"}) and for
any $y\in N^{n}$, we have 
\begin{equation*}
Hess(u_{j})(y)+\frac{\lambda _{\alpha }}{n}u_{j}(y)h=\sum_{i=1}^{N_{\alpha
}}\varphi _{i}^{\alpha }(m_{j})\left( Hess(\varphi _{i}^{\alpha })(y)+\frac{%
\lambda }{n}\varphi _{i}^{\alpha }(y)h\right)
\end{equation*}%
and by the homogeneity of the manifold $N^{n}$, for any $y\in N^{n}$ there
is an isometry $\sigma $ on $N^{n}$ such that $\sigma (y)=m_{j}$,
consequently%
\begin{equation*}
\sigma ^{\ast }\left( Hess(u_{j})\left( .\right) +\frac{\lambda _{\alpha }}{n%
}u_{j}\left( .\right) h\right) (y)=\sigma ^{\ast }\left(
\sum_{i=1}^{N_{\alpha }}\varphi _{i}^{\alpha }(m_{j})\left( Hess(\varphi
_{i}^{\alpha })(.)+\frac{\lambda _{\alpha }}{n}\varphi _{i}^{\alpha
}(.)h\right) \right) (y)
\end{equation*}%
\begin{equation*}
=\sum_{i=1}^{N_{\alpha }}\varphi _{i}^{\alpha }(m_{j})\left( Hess(\varphi
_{i}^{\alpha })(m_{j})+\frac{\lambda _{\alpha }}{n}\varphi _{i}^{\alpha
}(m_{j})h\right)
\end{equation*}%
and from the relation(\ref{13}), we get%
\begin{equation}
Hess(u_{j})(y)+\frac{\lambda _{\alpha }}{n}u_{j}(y)h=0  \label{13'}
\end{equation}%
so by the equality(\ref{9}), we deduce 
\begin{equation}
(L_{Z}h)(X,Y)=2Hess(u_{j})(X,Y)=-2(\frac{\lambda _{\alpha }}{n}u_{j})h(X,Y)
\label{14}
\end{equation}%
and since $h(X,Y)=0,$ we obtain 
\begin{equation*}
(L_{Z}h)(X,Y)=0\text{.}
\end{equation*}
\end{proof}

\section{Harmonic maps}

Let $\left\{ \frac{\partial }{\partial x_{\alpha }}\right\} _{\alpha
=1,...,m}$ be an orthonormal basis in a neighborhood of a point $x\in M^{m}$%
, we have, for the mapping $\psi $ defined by (\ref{2'}) 
\begin{equation*}
\text{trace}\nabla d\psi (x)=\sum_{\alpha =1}^{m}\left( \nabla _{\frac{%
\partial }{\partial x_{\alpha }}}^{\psi ^{-1}TN^{n}}d\psi (\frac{\partial }{%
\partial x_{\alpha }})-d\psi (\nabla _{\frac{\partial }{\partial x_{\alpha }}%
}^{M^{m}}\frac{\partial }{\partial x_{\alpha }})\right)
\end{equation*}%
\begin{equation*}
=\sum_{\alpha =1}^{m}\left( \nabla _{\frac{\partial }{\partial x_{\alpha }}%
}^{\psi ^{-1}TN^{n}}d\phi _{t_{o}}^{w}odf(\frac{\partial }{\partial
x_{\alpha }})-d\phi _{t_{o}}^{w}odf(\nabla _{\frac{\partial }{\partial
x_{\alpha }}}^{M^{m}}\frac{\partial }{\partial x_{\alpha }})\right)
\end{equation*}%
\begin{equation*}
=\sum_{\alpha =1}^{m}Hess\phi _{t_{o}}^{w}\left( df(\frac{\partial }{%
\partial x_{\alpha }}),df(\frac{\partial }{\partial x_{\alpha }})\right)
+d\phi _{t_{o}}^{w}\left( \sum_{\alpha =1}^{m}\nabla _{\frac{\partial }{%
\partial x_{\alpha }}}^{f^{-1}TN^{n}}df(\frac{\partial }{\partial x_{\alpha }%
})\right)
\end{equation*}%
\begin{equation*}
-d\phi _{t_{o}}^{w}odf\left( \sum_{\alpha =1}^{m}\nabla _{\frac{\partial }{%
\partial x_{\alpha }}}^{M^{m}}\frac{\partial }{\partial x_{\alpha }}\right)
\end{equation*}%
\begin{equation*}
=\sum_{\alpha =1}^{m}Hess\phi _{t_{o}}^{w}\left( df(\frac{\partial }{%
\partial x_{\alpha }}),df(\frac{\partial }{\partial x_{\alpha }})\right)
+d\phi _{t_{o}}^{w}\left( \tau _{g}(f)\right)
\end{equation*}%
where 
\begin{equation*}
Hess\phi _{t_{o}}^{w}(X,Y)=\nabla _{X}^{(\phi _{t_{o}}^{w})^{-1}TN^{n}}d\phi
_{t_{o}}^{w}(Y)-d\phi _{t_{o}}^{w}\left( \nabla _{X}^{N^{n}}Y\right) .
\end{equation*}%
$\;$So if $f$ is harmonic i.e. $\tau _{g}(f)=0,$ we get 
\begin{equation}
\tau _{g}(\psi )=\sum_{\alpha =1}^{m}Hess\phi _{t_{o}}^{w}\left( df(\frac{%
\partial }{\partial x_{\alpha }}),df(\frac{\partial }{\partial x_{\alpha }}%
)\right) .  \label{4}
\end{equation}

\subsection{Lower bound of the index}

Let $\Gamma (f)$ be the space of vector fields along the map $%
f:M^{m}\rightarrow N^{n}$ \ i.e. the sections of the pulled back bundle on $%
M^{m}$ induced by $f$ \ from the tangent $TN^{n}$ bundle on $N^{n}$.

The general formula of the second variation of the Energy functional in the
direction of the vector fields $w$ writes 
\begin{equation*}
\frac{d^{2}E(f_{t})}{dt^{2}}\mid _{t=0}=\int_{M^{m}}\left( \left\Vert \nabla
w\right\Vert _{f^{-1}TN}^{2}-trace_{M}\left\langle
R^{N^{n}}(df,w)w,df\right\rangle _{f^{-1}TN^{n}}\right) dv_{g}
\end{equation*}%
\begin{equation}
+\int_{M^{m}}\left\langle \nabla _{\frac{\partial }{\partial t}}\frac{%
\partial f}{\partial t},trace\nabla df\right\rangle _{f^{-1}TN^{n}}dv_{g}
\label{5}
\end{equation}%
and if $f$ \ is harmonic, we obtain 
\begin{equation}
\frac{d^{2}E(f_{t})}{dt^{2}}\mid _{t=0}=\int_{M^{m}}\left( \left\Vert \nabla
w\right\Vert _{f^{-1}TN^{n}}^{2}-trace_{M}\left\langle
R^{N^{n}}(df,w)w,df\right\rangle _{f^{-1}TN^{n}}\right) dv_{g}\text{.}
\label{6}
\end{equation}

For any vector field $w\;$on the target manifold $N^{n}$ along $f$, we
associate the following quadratic form 
\begin{equation}
Q_{f}(w)=\frac{d^{2}E_{g}(f_{t})}{dt^{2}}\mid _{t=0}  \label{7}
\end{equation}%
where $f_{t}(x)=\exp (tw)of(x)$. The Morse index of the harmonic map $f$ is
defined as the integer 
\begin{equation*}
Ind(f)=Sup\left\{ \dim F;F\subset \Gamma (f)\text{ such that }Q_{f}\text{\
is negative defined on }F\right\} \text{.}
\end{equation*}

Now let $S_{g}(f)$ be the stress-energy tensor introduced by Baird and Eells
(see $\left[ 1\right] $).%
\begin{equation}
S_{g}(f)=e_{g}(f)g-f^{\ast }h  \label{15}
\end{equation}%
For every $x\in M^{m},$ we put 
\begin{equation*}
S_{g}^{o}(f)(x)=Inf\left\{ S_{g}(f)(X,X):X\in T_{x}M^{m}\text{ and }%
g(X,X)=1\right\} .
\end{equation*}%
The tensor will be said positive at $x$ ( resp. positive definite ) if we
have $S_{g}^{o}(f)(x)\geq 0$ ( resp. $S_{g}^{o}(f)(x)>0$).

First we state the following theorem

\begin{theorem}
Let $(M^{m},g)$ be a Riemannian compact $m-$dimensional manifold and ($%
N^{n},h)\;$ be a homogenous strongly harmonic Riemannian manifold of
dimension $n\geq 3\;$with sectional curvature $K$ satisfying $K\geq \kappa
>0 $ where $\kappa $ is a positive constant. Let $f:M^{m}\rightarrow N^{n}$
\ be a non constant harmonic map. Suppose that the stress-energy tensor of $%
f\; $ is positive definite everywhere on $M^{m}$ and the nonvanishing
eigenvalue of the Laplacian operator $\lambda $ satisfies $\lambda \leq 
\frac{n^{2}}{2}\kappa .\;$Then the index of\ \ $f$ ,$\;Ind_{E}(f)\geq \dim
L, $ where $L$ is a $n+1$ -dimensional subspace of the eigenspace $%
V_{\lambda }$ corresponding to $\lambda $.
\end{theorem}

\begin{proof}
At each point $x\in M^{m}$, we denote respectively by $w^{T}(x)$ and $%
w^{\bot }(x)$ the tangential and the normal projections of $w(x)$ on the
space $df(T_{x}M)$ and $df(T_{x}M^{m})^{\bot }.$ Let $\left\{
e_{1},...,e_{m}\right\} $ be an orthonormal basis of $T_{x}M^{m}$ \ which
diagonalizes $f^{\ast }h$ such that $\left\{ df(e_{1}),...,df(e_{l})\right\} 
$ be a basis of $df(T_{x}M^{m})$. If\ $e_{g}(f)(x)\neq 0$, then at the point 
$x$ we have 
\begin{equation*}
\left\Vert w^{T}(x)\right\Vert _{h}^{2}=\sum_{i=1}^{l}\left\Vert
df\right\Vert _{h}^{-2}\left\langle w(x),df(e_{i})\right\rangle _{h}^{2}
\end{equation*}%
on the other hand we have for each $i\leq l$%
\begin{equation*}
\left\Vert df(e_{i})(x)\right\Vert
_{h}^{2}=e_{g}(f)(x)-S_{g}(f)(x)(e_{i},e_{i})
\end{equation*}%
\begin{equation*}
\leq e_{g}(f)(x)-S_{g}^{o}(f)(x).
\end{equation*}%
So we deduce that 
\begin{equation*}
\left( e_{g}(f)-S_{g}^{o}(f)\right) (x)\left\Vert w^{T}(x)\right\Vert
_{h}^{2}\geq \sum_{i=1}^{l}h\left( w(x),df(e_{i})\right) ^{2}
\end{equation*}%
and 
\begin{equation*}
\sum_{i=1}^{l}h\left( w(x),df(e_{i})\right) ^{2}-e_{g}(f)\left\Vert
w^{T}(x)\right\Vert _{h}^{2}\leq -S_{g}^{o}(f)(x)\left\Vert
w^{T}(x)\right\Vert _{h}^{2}
\end{equation*}%
then 
\begin{equation*}
\sum_{i=1}^{l}h\left( w(x),df(e_{i})\right) ^{2}-e_{g}(f)\left\Vert
w(x)\right\Vert _{h}^{2}
\end{equation*}%
\begin{equation*}
=\sum_{i=1}^{l}h\left( w(x),df(e_{i})\right) ^{2}-e_{g}(f)\left\Vert
w^{T}(x)\right\Vert _{h}^{2}-e_{g}(f)\left\Vert w^{\bot }(x)\right\Vert
_{h}^{2}
\end{equation*}%
\begin{equation*}
\leq -S_{g}^{o}(f)(x)\left\Vert w^{T}(x)\right\Vert
_{h}^{2}-e_{g}(f)\left\Vert w^{\bot }(x)\right\Vert _{h}^{2}
\end{equation*}%
\begin{equation*}
\leq -S_{g}^{o}(f)(x)\left\Vert w(x)\right\Vert _{h}^{2}.
\end{equation*}%
Consequently 
\begin{equation}
\sum_{i=1}^{l}h\left( w(x),df(e_{i})\right) ^{2}-e_{g}(f)\left\Vert
w(x)\right\Vert _{h}^{2}\leq -S_{g}^{o}(f)(x)\left\Vert w(x)\right\Vert
_{h}^{2}.  \label{16}
\end{equation}%
Now we let 
\begin{equation*}
w_{j}(x)=-gradu_{j}(f(x))
\end{equation*}%
where $u_{j}=\sum_{i=1}^{N_{\alpha }}$ $\varphi _{i}^{\alpha }(m_{j})\varphi
_{i}^{\alpha }$ is an eigenfunction of the Laplacian operator on $N^{n}$
defined previously by (\ref{8''}). Then 
\begin{equation*}
\left\Vert df\right\Vert _{f^{-1}\NEG{T}N^{n}}^{2}\left\Vert
gradu_{j}\right\Vert _{f^{-1}\NEG{T}N^{n}}^{2}-\sum_{i=1}^{l}\left\langle
df(e_{i}),gradu_{j}\right\rangle _{f^{-1}\NEG{T}N^{n}}^{2}
\end{equation*}%
\begin{equation*}
\geq S_{g}^{o}(f)\left\Vert gradu_{j}\right\Vert _{f^{-1}\NEG%
{T}N^{n}}^{2}+e_{g}(f)\left\Vert gradu_{j}\right\Vert _{f^{-1}\NEG%
{T}N^{n}}^{2}.
\end{equation*}%
So, we have 
\begin{equation*}
\left\Vert \nabla gradu_{j}\right\Vert
_{f^{-1}TN^{n}}^{2}-\sum_{i=1}^{l}\left\langle
R^{N^{n}}(df(e_{i}),gradu_{j})df(e_{i}),gradu_{j}\right\rangle _{f^{-1}\NEG%
{T}N^{n}}\leq
\end{equation*}%
\begin{equation*}
\left\Vert \nabla gradu_{j}\right\Vert _{f^{-1}TN^{n}}^{2}-\kappa
\sum_{i=1}^{l}\left( \left\Vert df(e_{i})(x)\right\Vert _{h}^{2}\left\Vert
gradu_{j}\right\Vert _{f^{-1}\NEG{T}N^{n}}^{2}-\left\langle
df(e_{i}),gradu_{j}\right\rangle _{f^{-1}\NEG{T}N^{n}}^{2}\right) \leq
\end{equation*}%
\begin{equation*}
\left\Vert \nabla gradu_{j}\right\Vert _{f^{-1}TN^{n}}^{2}-\kappa \left(
e_{g}(f)\left\Vert gradu_{j}\right\Vert _{f^{-1}\NEG%
{T}N^{n}}^{2}+S_{g}^{o}(f)\left\Vert gradu_{j}\right\Vert _{f^{-1}\NEG%
{T}N^{n}}^{2}\right) .
\end{equation*}%
Since 
\begin{equation*}
\left\Vert \nabla gradu_{j}\right\Vert _{f^{-1}TN^{n}}^{2}=g^{\alpha \beta
}\left\langle \nabla _{\frac{\partial }{\partial x_{\alpha }}%
}gradu_{j},\nabla _{\frac{\partial }{\partial x_{\beta }}}gradu_{j}\right%
\rangle _{f^{-1}\NEG{T}N^{n}}
\end{equation*}%
\begin{equation*}
=g^{\alpha \beta }\frac{\partial f^{i}}{\partial x_{\alpha }}\frac{\partial
f^{k}}{\partial x_{\beta }}\left\langle \nabla _{\frac{\partial }{\partial
f^{i}}}gradu_{j},\nabla _{\frac{\partial }{\partial f^{k}}%
}gradu_{j}\right\rangle _{h}
\end{equation*}%
\begin{equation*}
=g^{\alpha \beta }\frac{\partial f^{i}}{\partial x_{\alpha }}\frac{\partial
f^{k}}{\partial x_{\beta }}hess(u_{j}(f(x)))\left\langle \frac{\partial }{%
\partial f^{i}},\nabla _{\frac{\partial }{\partial f^{k}}}gradu_{j}\right%
\rangle _{h}
\end{equation*}%
where the Einstein convention summation is used,and taking account of the
formula(\ref{13'})%
\begin{equation*}
=-g^{\alpha \beta }\frac{\partial f^{i}}{\partial x_{\alpha }}\frac{\partial
f^{k}}{\partial x_{\beta }}\frac{\lambda }{n}u_{j}(f(x))\left\langle \frac{%
\partial }{\partial f^{i}},\nabla _{\frac{\partial }{\partial f^{k}}%
}gradu_{j}\right\rangle _{h}
\end{equation*}%
\begin{equation*}
=g^{\alpha \beta }\frac{\partial f^{i}}{\partial x_{\alpha }}\frac{\partial
f^{k}}{\partial x_{\beta }}\frac{\lambda ^{2}}{n^{2}}u_{j}^{2}h_{ik}=2\frac{%
\lambda ^{2}}{n^{2}}e_{g}(f)u_{j}^{2}.
\end{equation*}%
On the other hand, we get from formula (\ref{11}) 
\begin{equation*}
\int_{M^{m}}\left\Vert gradu_{j}\right\Vert
_{f^{-1}TN^{n}}^{2}dv_{g}=\int_{M^{m}}\left\Vert du_{j}\right\Vert
_{f^{-1}TN^{n}}^{2}dv_{g}=\int_{M^{m}}trace(du_{j}\otimes du_{j})(f(x))dv_{g}
\end{equation*}%
\begin{equation*}
=\lambda \int_{M^{m}}u_{j}^{2}(f(x))dv_{g}
\end{equation*}%
so 
\begin{equation}
\int_{M^{m}}\left\Vert \nabla gradu_{j}\right\Vert
_{f^{-1}TN^{n}}^{2}dv_{g}=2\frac{\lambda }{n^{2}}\int_{M^{m}}e_{g}(f)\left%
\Vert gradu_{j}\right\Vert _{h}^{2}dv_{g}\text{.}  \label{17}
\end{equation}%
Consequently 
\begin{equation*}
\int_{M^{m}}\left( \left\Vert \nabla gradu_{j}(f(x))\right\Vert
_{h}^{2}-\sum_{i=1}^{m}\left\langle
R^{N^{n}}(df(e_{i}),gradu_{j})df(e_{i}),gradu_{j}(f(x))\right\rangle
_{h}\right) dv_{g}\leq
\end{equation*}%
\begin{equation*}
\int_{M^{m}}\left( 2\frac{\lambda }{n^{2}}e_{g}(f)\left\Vert
gradu_{j}(f(x))\right\Vert _{h}^{2}-\kappa (e_{g}(f)\left\Vert
gradu_{j}(f(x))\right\Vert _{h}^{2}+S_{g}^{o}(f)\left\Vert
gradu_{j}(f(x))\right\Vert _{h}^{2})\right) dv_{g}=
\end{equation*}%
\begin{equation*}
(2\frac{\lambda }{n^{2}}-\kappa )\int_{M^{m}}e_{g}(f)\left\Vert
gradu_{j}(f(x))\right\Vert _{h}^{2}dv_{g}-\kappa
\int_{M^{m}}S_{g}^{o}(f)\left\Vert gradu_{j}(f(x))\right\Vert _{h}^{2}dv_{g}
\end{equation*}%
Hence 
\begin{equation*}
\frac{d^{2}E(f_{t})}{dt^{2}}\mid _{t=0}\leq
\end{equation*}%
\begin{equation}
\int_{M^{m}}e_{g}(f)(2\frac{\lambda }{n^{2}}-\kappa )\left\Vert
gradu_{j}(f(x))\right\Vert _{h}^{2}dv_{g}-\kappa
\int_{M^{m}}S_{g}^{o}(f)(x)\left\Vert gradu_{j}(f(x))\right\Vert
_{h}^{2}dv_{g}  \label{18}
\end{equation}%
which shows that if the eigenvalue $\lambda $ of the Laplacian $\Delta
_{N^{n}}$ operator on $N^{n}$ satisfies 
\begin{equation*}
\lambda \leq \frac{n^{2}\kappa }{2}\;
\end{equation*}%
then 
\begin{equation*}
\frac{d^{2}E(f_{t})}{dt^{2}}\mid _{t=0}\leq -\kappa
\int_{M^{m}}S_{g}^{o}(f)(x)\left\Vert gradu_{j}(f(x))\right\Vert
_{h}^{2}dv_{g}
\end{equation*}%
so 
\begin{equation*}
Ind_{E}(f)\geq \dim L
\end{equation*}%
where $L$ is the $n+1$ -dimensional subspace spanned by the gradient vector
fields $gradu_{j}$, $j=1,...,n+1$.
\end{proof}

\section{Harmonic map as a global maximum}

In this section, we prove the following global theorem

\begin{theorem}
\label{tg2} Let $(M^{m},g)$ be a compact Riemannian $m-$dimensional
manifold, $(N^{n},h)$ be a homogenous strongly harmonic Riemannian manifold
of dimension $n\geq 3$. If the stress-energy tensor of the harmonic map $%
f:M^{m}\rightarrow N^{n}\;$ is positive everywhere on $M^{m}$, then the map $%
f$ is a global maximum of the energy functional on the $n+1$- dimensional
subspace $L$ of the eigenspace $V_{\lambda }$ corresponding to a
nonvanishing eigenvalue $\lambda $ of the Laplacian operator on the target
manifold.
\end{theorem}

\begin{proof}
Let $\widetilde{\nabla }^{N}$ be the Levi-Civita connection on $\left(
N^{n},(\phi _{t_{o}}^{w})^{\ast }h\right) ,$ and $\left\{ \frac{\partial }{%
\partial x_{\alpha }}\right\} _{\alpha =1,...,m}$ be an orthonormal basis in
a neighborhood of a point $x\in M^{m}$ which diagonalizes $f^{\ast }h$ and
such that $(\;df(\frac{\partial }{\partial x_{1}}),...,df(\frac{\partial }{%
\partial x_{l}}))$ is a basis of $df(T_{x}M))$. Then 
\begin{equation*}
d\phi _{t_{o}}^{w}(\widetilde{\nabla }_{X}^{N^{n}}Y)=\nabla _{d\phi
_{t_{o}}^{w}(X)}^{N^{n}}d\phi _{t_{o}}^{w}(Y)
\end{equation*}%
where for simplicity $w=w_{j}$ defined previously as minus the gradient of
the eigenfunction $u_{j}$ given by (\ref{8''}), so 
\begin{equation*}
\tau _{g}(\psi )=\sum_{\alpha =1}^{m}Hess\phi _{t_{o}}^{w}\left( df(\frac{%
\partial }{\partial x_{\alpha }}),df(\frac{\partial }{\partial x_{\alpha }}%
)\right)
\end{equation*}%
\begin{equation*}
=\sum_{\alpha =1}^{m}d\phi _{t_{o}}^{w}\left( \widetilde{\nabla }_{\frac{%
\partial }{\partial x_{\alpha }}}^{f^{-1}TN^{n}}df(\frac{\partial }{\partial
x_{\alpha }})-\nabla _{\frac{\partial }{\partial x_{\alpha }}%
}^{f^{-1}TN^{n}}df(\frac{\partial }{\partial x_{\alpha }})\right) \text{.}
\end{equation*}%
Since $w$ is a conformal infinitesimal transformation i.e. $(\phi
_{t_{o}}^{w})^{\ast }h=e^{2\sigma }h$, where $\sigma $ is some function on
the manifold $N^{n}$, the conformal change of connections is then given by:
for any vector field $X$, $Y$ on the manifold $N^{n}$ 
\begin{equation*}
\widetilde{\nabla }_{X}Y-\nabla _{X}Y=X(\sigma )Y+Y(\sigma )X-\left\langle
X,Y\right\rangle _{h}\nabla \sigma
\end{equation*}%
where $\;\left\langle .,.\right\rangle _{h}=h\;$and\ $\nabla \sigma $ is the
gradient vector field of the function $\sigma .\;$Hence 
\begin{equation*}
\tau _{g}(\psi )=d\phi _{t_{o}}^{w}\left( 2\sum_{\alpha =1}^{m}\left\langle
\nabla \sigma of,df(\frac{\partial }{\partial x_{\alpha }})\right\rangle
_{h}df(\frac{\partial }{\partial x_{\alpha }})-\sum_{\alpha
=1}^{m}\left\langle df(\frac{\partial }{\partial x_{\alpha }}),df(\frac{%
\partial }{\partial x_{\alpha }})\right\rangle _{h}\nabla \sigma of\right)
\end{equation*}%
\begin{equation*}
=2d\phi _{t_{o}}^{w}\left( \sum_{\alpha =1}^{m}\left\langle \nabla \sigma
of,df(\frac{\partial }{\partial x_{\alpha }})\right\rangle _{h}df(\frac{%
\partial }{\partial x_{\alpha }})-e_{g}(f).\nabla \sigma of\right)
\end{equation*}%
where $\psi =\phi _{t_{o}}^{w}of$.

Consequently%
\begin{equation*}
\int_{M}\left\langle \tau _{g}(\psi ),wo\psi \right\rangle _{h}dv_{g}=
\end{equation*}%
\begin{equation*}
=2\int_{M}(\phi _{t_{o}}^{w})^{\ast }\left( \sum_{\alpha =1}^{m}\left\langle
\nabla \sigma of,df(\frac{\partial }{\partial x_{\alpha }})\right\rangle
_{h}\left\langle df(\frac{\partial }{\partial x_{\alpha }},wof\right\rangle
_{h}-e_{g}(f)\left\langle \nabla \sigma of,w\right\rangle _{h}\right) dv_{g}
\end{equation*}%
\begin{equation*}
=2\int_{M}e^{2\sigma of(x)}\left\{ \sum_{\alpha =1}^{m}\left\langle \nabla
\sigma of,df(\frac{\partial }{\partial x_{\alpha }})\right\rangle
_{h}\left\langle df(\frac{\partial }{\partial x_{\alpha }}),wof\right\rangle
_{h}-e_{g}(f)\left\langle \nabla \sigma of,wof\right\rangle _{h}\right\}
dv_{g}
\end{equation*}%
Now we compute the gradient of the function $\sigma $ to get 
\begin{equation*}
\left\langle \nabla \sigma of,wof\right\rangle _{h}(x)=\frac{\left\langle
\nabla _{wof}w(\phi _{t_{0}}^{w}(f(x)),w(\phi
_{t_{0}}^{w}(f(x))\right\rangle _{h}}{\left\Vert w(\phi
_{t_{0}}^{w}(f(x))\right\Vert _{h}^{2}}-\frac{\left\langle \nabla
_{wof}w(f(x)),w(f(x))\right\rangle _{h}}{\left\Vert w(f(x))\right\Vert
_{h}^{2}}
\end{equation*}%
\begin{equation*}
=-\frac{Hess(u(\phi _{t_{0}}^{w}(f(x)))\left\langle w(\phi
_{t_{0}}^{w}(f(x)),w(\phi _{t_{0}}^{w}(f(x))\right\rangle _{h}}{\left\Vert
w(\phi _{t_{0}}^{w}(f(x))\right\Vert _{h}^{2}}+\frac{Hess(u(f(x)))\left%
\langle w(f(x)),w(f(x))\right\rangle _{h}}{\left\Vert w(f(x))\right\Vert
_{h}^{2}}.
\end{equation*}%
Using the inequality(\ref{13'}), we get that 
\begin{equation*}
\left\langle \nabla \sigma of,wof\right\rangle _{h}(x)=\frac{\lambda }{n}%
\left( u(\phi _{t_{0}}^{w}(f(x))-u((f(x))\right)
\end{equation*}%
and also, we have%
\begin{equation*}
\left\langle \nabla \sigma of,df(\frac{\partial }{\partial x_{\alpha }}%
)\right\rangle _{h}(x)=
\end{equation*}%
\begin{equation*}
-\frac{Hess(u(\phi _{t_{0}}^{w}(f(x)))\left\langle \left( \phi
_{t_{0}}^{w}\right) _{\ast }df(\frac{\partial }{\partial x_{\alpha }}%
),w(\phi _{t_{0}}^{w}(f(x))\right\rangle _{h}}{\left\Vert w(\phi
_{t_{0}}^{w}(f(x))\right\Vert _{h}^{2}}+\frac{Hess(u(f(x)))\left\langle df(%
\frac{\partial }{\partial x_{\alpha }}),w(f(x))\right\rangle _{h}}{%
\left\Vert w(f(x))\right\Vert _{h}^{2}}
\end{equation*}%
\begin{equation*}
=\frac{\lambda }{n}\left( u(\phi _{t_{0}}^{w}(f(x))\frac{\left\langle \left(
\phi _{t_{0}}^{w}\right) _{\ast }df(\frac{\partial }{\partial x_{\alpha }}%
),w(\phi _{t_{0}}^{w}(f(x))\right\rangle _{h}}{\left\Vert w(\phi
_{t_{0}}^{w}(f(x))\right\Vert _{h}^{2}}-u((f(x))\frac{\left\langle df(\frac{%
\partial }{\partial x_{\alpha }}),w(f(x))\right\rangle _{h}}{\left\Vert
w(f(x))\right\Vert _{h}^{2}}\right)
\end{equation*}%
\begin{equation*}
=\frac{\lambda }{n}\left( u(\phi _{t_{0}}^{w}(f(x))-u((f(x))\right) \frac{%
\left\langle df(\frac{\partial }{\partial x_{\alpha }}),w(f(x))\right\rangle
_{h}}{\left\Vert w(f(x))\right\Vert _{h}^{2}}
\end{equation*}%
consequently%
\begin{equation*}
\left\langle \nabla \sigma of,df(\frac{\partial }{\partial x_{\alpha }}%
)\right\rangle _{h}(x)\left\langle df(\frac{\partial }{\partial x_{\alpha }}%
),wof\right\rangle _{h}(x)=\frac{\lambda }{n}\left( u(\phi
_{t_{0}}^{w}(f(x))-u((f(x))\right) \frac{\left\langle df(\frac{\partial }{%
\partial x_{\alpha }}),w(f(x))\right\rangle _{h}^{2}}{\left\Vert
w(f(x))\right\Vert _{h}^{2}}
\end{equation*}%
and%
\begin{equation*}
\sum_{\alpha =1}^{m}\left\langle \nabla \sigma of,df(\frac{\partial }{%
\partial x_{\alpha }})\right\rangle _{h}(x)\left\langle df(\frac{\partial }{%
\partial x_{\alpha }}),wof\right\rangle _{h}(x)-e_{g}(f)\left\langle \nabla
\sigma of,wof\right\rangle _{h}(x)=
\end{equation*}%
\begin{equation*}
=\frac{\lambda }{n}\frac{\left( u(\phi _{t_{0}}^{w}(f(x))-u((f(x))\right) }{%
\left\Vert w(f(x))\right\Vert _{h}^{2}}\left( \sum_{\alpha
=1}^{m}\left\langle df(\frac{\partial }{\partial x_{\alpha }}%
),w(f(x))\right\rangle _{h}^{2}-e_{g}(f)\left\Vert w(f(x))\right\Vert
_{h}^{2}\right) .
\end{equation*}%
Now, if at each point $x\in M^{m}$, $w^{T}(x)$ and $w^{\bot }(x)$ denote the
tangential and the normal projections of $w(x)$ on the space $df(T_{x}M)$
and $df(T_{x}M^{m})^{\bot }$ and $S_{g}^{o}(f)$ is the stress-energy tensor
of $f$, we have%
\begin{equation*}
\left( e_{g}(f)-S_{g}^{o}(f)\right) \left\Vert w^{T}(f(x))\right\Vert
_{h}^{2}=\left( e_{g}(f)-S_{g}^{o}(f)\right) \sum_{\alpha
=1}^{l}\left\langle df(\frac{\partial }{\partial x_{\alpha }}),\frac{w(f(x))%
}{\left\Vert df(\frac{\partial }{\partial x_{\alpha }})\right\Vert _{h}}%
\right\rangle _{h}^{2}
\end{equation*}%
\begin{equation*}
\geq \sum_{\alpha =1}^{l}\left\langle df(\frac{\partial }{\partial x_{\alpha
}}),w(f(x))\right\rangle _{h}^{2}
\end{equation*}%
so%
\begin{equation*}
\sum_{\alpha =1}^{m}\left\langle df(\frac{\partial }{\partial x_{\alpha }}%
),w(f(x))\right\rangle _{h}^{2}-e_{g}(f)\left\Vert w(f(x))\right\Vert
_{h}^{2}\leq
\end{equation*}%
\begin{equation*}
\leq \sum_{\alpha =l+1}^{m}\left\langle df(\frac{\partial }{\partial
x_{\alpha }}),w(f(x))\right\rangle _{h}^{2}-S_{g}^{o}(f)\left\Vert
w^{T}(f(x))\right\Vert _{h}^{2}-e_{g}(f)\left\Vert w^{\bot
}(f(x))\right\Vert _{h}^{2}
\end{equation*}%
\begin{equation*}
\leq \sum_{\alpha =l+1}^{m}\left\langle df(\frac{\partial }{\partial
x_{\alpha }}),w(f(x))\right\rangle _{h}^{2}+\left(
e_{g}(f)-S_{g}^{o}(f)\right) \left\Vert w^{T}(f(x))\right\Vert
_{h}^{2}-e_{g}(f)\left\Vert w(f(x))\right\Vert _{h}^{2}
\end{equation*}%
and since $S_{g}^{o}(f)\leq e_{g}(f)$, we obtain 
\begin{equation*}
\sum_{\alpha =1}^{m}\left\langle df(\frac{\partial }{\partial x_{\alpha }}%
),w(f(x))\right\rangle _{h}^{2}-e_{g}(f)\left\Vert w(f(x))\right\Vert
_{h}^{2}\leq -S_{g}^{o}(f)\left\Vert w(f(x))\right\Vert _{h}^{2}
\end{equation*}

and since the vector field $w$ is minus the gradient of the eigenfunction$%
\;u $ corresponding to eigenvalue $\lambda $, we get 
\begin{equation*}
\frac{d}{dt}\mid _{t=t_{o}}u(\phi _{t}^{w}(f(x))=-\left\Vert w(\phi
_{t_{0}}^{w}(f(x))\right\Vert \leq 0
\end{equation*}%
that is the function $t\rightarrow u(\phi _{t}^{w}(f(x))$ is decreasing.

Consequently%
\begin{equation*}
\sum_{\alpha =1}^{m}\left\langle \nabla \sigma of,df(\frac{\partial }{%
\partial x_{\alpha }})\right\rangle _{h}(x)\left\langle df(\frac{\partial }{%
\partial x_{\alpha }}),wof\right\rangle _{h}(x)-e_{g}(f)\left\langle \nabla
\sigma of,wof\right\rangle _{h}(x)\geq
\end{equation*}%
\begin{equation*}
-\frac{\lambda }{n}\left( u(\phi _{t_{0}}^{w}(f(x))-u((f(x))\right)
S_{g}^{o}(f)
\end{equation*}%
and finally we obtain that 
\begin{equation*}
\frac{d}{dt}E(f_{t})\mid _{t=t_{0}}=-\int_{M^{m}}\left\langle \tau _{g}(\psi
),wo\psi \right\rangle _{f^{-1}TN}dv_{g}
\end{equation*}%
\begin{equation*}
\leq \frac{\lambda }{n}\int_{M^{m}}e^{2\sigma of(x)}\left( u(\phi
_{t_{0}}^{w}(f(x))-u((f(x))\right) S_{g}^{o}(f)dv_{g}
\end{equation*}%
So if the stress-energy tensor $S_{g}^{o}(f)$ is positive, the energy
functional is decreasing that means that the harmonic map $f$ is a global
maximum of the energy functional $E(.).$
\end{proof}

Let $Isom(S^{n})$ and $Conf(S^{n})$ be respectively the isometric and the
conformal Lie group on the standard unit sphere $S^{n}$.$\;$Denote by $%
conf(S^{n})$ the Lie algebra of $Conf(S^{n})$ and let $w$ be the negative
gradient of the eigenfunction $u_{j}$ $=\sum_{i=1}^{N_{\alpha }}\varphi
_{i}(m_{j})\varphi _{i}$ corresponding to a non vanishing eigenvalue of the
Laplacian on the $S^{n}$, given by (\ref{8"}), we have

\begin{corollary}
If the target manifold is the standard unit $n$- sphere $S^{n}$ ($n\geq 3)\;$
endowed with canonical metric then for every $w\in conf(S^{n})$

\begin{equation*}
E(wof)\leq E(f)
\end{equation*}%
provided that the stress-energy tensor is positive.
\end{corollary}

\begin{proof}
We know that 
\begin{equation*}
\dim (Conf(S^{n})/Isom(S^{n}))=n+1.
\end{equation*}%
Since the eigenfunctions, corresponding to the first non zero eigenvalue $%
\lambda _{1}=n$ of the Laplacian operator on the standard $n$- sphere $S^{n}$%
, are the restrictions of the linear forms on the Euclidean space $R^{n+1}$
to $S^{n}$ and their gradients are conformal vector fields, it follows that
the set of these gradients is nothing than $conf(S^{n})$ and by Theorem\ref%
{tg2}, we get for every $w\in conf(S^{n})$ :$\;E(wof)\leq E(f).$
\end{proof}

\begin{remark}
In the particular case where the target manifold is the standard n- sphere,
we have $\;$ $\dim V_{\lambda }=n+1$.$\;$So we get the result in (\cite{5}).
\end{remark}

\section{Minimal Immersions}

\subsection{Morse index of minimal immersions}

With notations of the previous section, if the manifold $M^{m}$ is a minimal
submanifold of the manifold $N^{n}$, the second variation with respect to
the vector field $w$ along $f$ is given by the following integral 
\begin{equation*}
\frac{d^{2}}{dt^{2}}V(f_{t})\mid _{t=0}=
\end{equation*}%
\begin{equation*}
=\int_{M^{m}}\left( \left\Vert \nabla ^{f^{-1}TN^{n}}wof\right\Vert
_{h}^{2}-\left\Vert \sigma (wof)\right\Vert
_{h}^{2}-trace_{M^{m}}\left\langle R^{N^{n}}(df,wof)wof,df\right\rangle
_{h}\right) dv_{g}
\end{equation*}%
where 
\begin{equation*}
trace_{M^{m}}\left\langle R^{N^{n}}(df,w)w,df\right\rangle
=\sum_{i=1}^{m}\left\langle R^{N^{n}}(df(\frac{\partial }{\partial x_{i}}%
),w)w,df(\frac{\partial }{\partial x_{i}})\right\rangle _{h}
\end{equation*}%
and $\sigma (w)$ denotes the second fundamental form relative to $w.$

To the immersion $f$, we assign the quadratic form associated to the second
variation of the volume functional defined on $\Gamma (f)$ by 
\begin{equation*}
H_{f}(w)=\frac{d^{2}}{dt^{2}}V(f_{t})\mid _{t=0}
\end{equation*}%
and since this latter depends only on the normal component of elements of $%
\Gamma (f),$ we consider only the restriction of $H_{f}$ to the normal
projection $\Gamma ^{N^{n}}(f)$ of the space $\Gamma (f).$ The Morse index
of $f,$ denoted $Ind_{V}(f),$ is defined as the dimension of the maximal
subspace of $\Gamma ^{N}(f)$ on which $H_{f}$ is negative- definite.

In this section, we state the following theorem

\begin{theorem}
Let $(M^{m},g)$ be a Riemannian compact $m$- dimensional manifold, $%
(N^{n},h) $ be a homogenous strongly harmonic Riemannian manifold of
dimension $n\geq 3 $\ with sectional curvature $K$ satisfying $K\geq \kappa
>0$, where $\kappa $ is a constant, $f:M^{m}\rightarrow N^{n}$ be a minimal
isometric immersion not totally geodesic and $\lambda $ be a non vanishing
eigenvalue of the Laplacian operator on the target manifold $N^{n}$. Suppose
that $\lambda $ satisfies $\lambda \leq \frac{n^{2}}{2}\kappa $.\ Then $%
Ind_{V}(f)\geq \dim (L{}^{\bot })$ where $\ L{}^{\bot }$ denotes the normal
component of \ an $n+1$- dimensional subspace $L$ of the eigenspace $%
V_{\lambda }$ corresponding to $\lambda $. \ 
\end{theorem}

\begin{proof}
$H_{f}(w)$ can be written as 
\begin{equation*}
H_{f}(w)=\int_{M^{m}}\left( \left\Vert \nabla ^{f^{-1}TN^{n}}w^{\perp
}of\right\Vert _{h}^{2}-trace_{M_{m}}\left\langle R^{N^{n}}(df,w^{\perp
}of)w^{\perp }of,df\right\rangle _{h}\right) dv_{g}
\end{equation*}%
\begin{equation*}
-\int_{M}\left\Vert \sigma (w^{\perp }of)\right\Vert ^{2}dv_{g}.
\end{equation*}%
where $w=-gradu_{j}$ and $w^{\bot }$ is the orthogonal projection of $w$ on $%
df(T_{x}M^{m}).$ Now taking account of formula(\ref{17}), we obtain 
\begin{equation*}
H_{f}(w)=\int_{M}2e_{g}(f)\left( \frac{\lambda }{n^{2}}-\frac{1}{2}\kappa
\right) \left\Vert w^{\perp }of\right\Vert _{h}^{2}dv_{g}-\int_{M}\left\Vert
\sigma (w^{\perp }of)\right\Vert _{h}^{2}dv_{g}.
\end{equation*}%
so, since 
\begin{equation*}
\lambda \leq \frac{\kappa n^{2}}{2}
\end{equation*}%
and 
\begin{equation*}
\sigma (w^{\perp })\neq 0
\end{equation*}%
we obtain 
\begin{equation*}
H_{f}(w)<0.
\end{equation*}
\end{proof}

\subsection{Minimal Immersion as a global maximum}

In this subsection, we establish the following global theorem

\begin{theorem}
Let $(M^{m},g)$ be a Riemannian $m-$dimensional compact manifold, $(N^{n},h)$
be a strongly harmonic Riemannian manifold of dimension $n\geq 3$. If $%
f:M^{m}\rightarrow N^{n}\;$is a minimal isometric immersion, then $f$ is a
global maximum of the volume functional on the normal component $L^{\perp }$
of an $n+1$- dimensional subspace $L$ of the eigenspace $V_{\lambda }$
corresponding to a nonvanishing eigenvalue of the Laplacian operator $%
\lambda $ on the target manifold.
\end{theorem}

\begin{proof}
Let $w=-gradu_{j},$ where $u_{j}$ is the eigenfunction of the Laplacian
operator on $N^{n}$ given in section2 by (\ref{8''}). Let $\psi =\phi
_{t}^{w}of$ where $\phi _{t}^{w}$ denotes the flow generated by the vector
field $w.$ The first variation formula reads as 
\begin{equation*}
\frac{d}{dt}V(f_{t})=-\int_{\psi (M)}\left\langle H_{f(x)}^{\psi
(M^{m})},w(f(x))\right\rangle _{h}dv_{g}
\end{equation*}%
\begin{equation*}
=-\int_{M^{m}}\left\langle H_{\psi (x)}^{\psi (M^{m})},w(\psi
(x))\right\rangle _{h}e^{m\sigma }dv_{g}
\end{equation*}%
where $H^{\psi (M^{m})}$ denotes the mean curvature of the submanifold $\psi
(M^{m})$ in $N^{n}$ . The variation depends only on the normal component $%
w^{\bot }$ of the vector field $w$. Let $\left\{ \frac{\partial }{\partial
x_{1}},...,\frac{\partial }{\partial x_{m}}\right\} $ be an orthonormal
basis at the point $x\in M,$ then 
\begin{equation*}
\left\langle H_{\psi (x)}^{\psi (M^{m})},w(\psi (x))\right\rangle
_{h}=\sum_{i=1}^{m}\left\langle \nabla _{\psi _{\ast }\frac{\partial }{%
\partial x_{i}}}^{N^{n}}\psi _{\ast }\frac{\partial }{\partial x_{i}}%
,w^{\bot }(\psi _{t}(x))\right\rangle _{h}\text{.}
\end{equation*}%
Since $\phi _{t}^{w}:(N^{n},e^{2\sigma }h)\rightarrow (N^{n},h)$ is an
isometry, it follows that 
\begin{equation*}
\left( \phi _{t}^{w}\right) ^{\ast }(\overset{\_}{\nabla }_{f_{\ast }\frac{%
\partial }{\partial x_{i}}}^{N^{n}}f_{\ast }\frac{\partial }{\partial x_{i}}%
)=\nabla _{\left( \psi _{t}^{w}\right) _{\ast }\frac{\partial }{\partial
x_{i}}}^{N^{n}}\left( \psi _{t}^{w}\right) _{\ast }\frac{\partial }{\partial
x_{i}}
\end{equation*}%
where $\overset{\_}{\nabla }^{N^{n}}$ is the connection on the manifold $%
(N^{n},e^{2\sigma }h)$ and $\psi ^{w}=\phi _{t}^{w}of.$ From the conformal
change of connections formula, we obtain 
\begin{equation*}
\left\langle H_{\psi (x)}^{\psi (M^{m})},w(\psi (x))\right\rangle
=\sum_{i=1}^{m}\left( \phi _{t}^{w}\right) ^{\ast }\left\langle \overset{\_}{%
\nabla }_{f_{\ast }\frac{\partial }{\partial x_{i}}}^{N^{n}}f_{\ast }\frac{%
\partial }{\partial x_{i}},w^{\bot }(f(x))\right\rangle
\end{equation*}%
\begin{equation*}
=\sum_{i=1}^{m}\left( \phi _{t}^{w}\right) ^{\ast }\left\langle \nabla
_{f_{\ast }\frac{\partial }{\partial x_{i}}}^{N^{n}}f_{\ast }\frac{\partial 
}{\partial x_{i}}+2\left\langle f_{\ast }\frac{\partial }{\partial x_{i}}%
,grad(\sigma )of\right\rangle f_{\ast }\frac{\partial }{\partial x_{i}}%
\right.
\end{equation*}%
\begin{equation*}
\left. -\left\langle f_{\ast }\frac{\partial }{\partial x_{i}},f_{\ast }%
\frac{\partial }{\partial x_{i}}\right\rangle grad(\sigma )of,w^{\bot
}of\right\rangle
\end{equation*}%
and since $\left\langle f_{\ast }\frac{\partial }{\partial x_{i}},w^{\bot
}of\right\rangle =0$, we obtain 
\begin{equation*}
\left\langle H_{\psi (x)}^{\psi (M^{m})},w(\psi (x))\right\rangle =\left(
\phi _{t}^{w}\right) ^{\ast }\left\langle H_{f(x)}^{f(M^{m})}-mgrad(\sigma
)of,wof\right\rangle
\end{equation*}%
\begin{equation*}
=-m\left( \phi _{t}^{w}\right) ^{\ast }\left\langle grad(\sigma
)of,wof\right\rangle \text{, since }f\text{ is minimal.}
\end{equation*}%
Consequently 
\begin{equation*}
\frac{d}{dt}\mid _{t=t_{o}}V(f_{t})=m\int_{M^{m}}e^{2\sigma
of(x)}\left\langle grad(\sigma )of,wof\right\rangle dv_{g};\;\;\;\;t_{o}>0
\end{equation*}%
and taking into account the relation (\ref{13'})%
\begin{equation*}
\frac{d}{dt}\mid _{t=t_{o}}V(f_{t})=\frac{m\lambda }{n}\int_{M^{m}}e^{2%
\sigma of(x)}\left( u(\phi _{2t_{0}}^{w}(f(x))-u(\phi
_{t_{0}}^{w}f(x)\right) dv_{g}\text{.}
\end{equation*}%
Since the vector field $w$ is minus the gradient of the eigenfunction $u_{j}$%
, we have 
\begin{equation*}
\frac{d}{dt}u_{j}(\phi _{t}^{w}of(x))_{t=t_{o}}=-\left\Vert w\left( \phi
_{t_{o}}^{w}(f(x))\right) \right\Vert _{h}^{2}\leq 0
\end{equation*}%
so$\;\frac{d}{dt}V(f_{t})\leq 0\;$ that is to say 
\begin{equation*}
V(f_{t})\leq V(M^{m})\text{.}
\end{equation*}
\end{proof}

\end{document}